\documentclass[11pt]{article}

\usepackage{amssymb,amsmath,amsthm}
\usepackage{epsfig}
\usepackage{breqn}
\usepackage{rotating}
\usepackage{amsmath}
\usepackage{amsfonts}
\usepackage{setspace}
\usepackage{fullpage}
\usepackage{enumitem}
\usepackage{bbold} 
\usepackage{comment}
\usepackage{tikz}
\usepackage{pgf}
\usetikzlibrary{positioning,arrows,shapes,decorations.markings,decorations.pathreplacing,matrix}
\tikzstyle{vertex}=[circle,draw=black,fill=black,inner sep=0,minimum size=3pt,text=white,font=\footnotesize]
\usepackage{mathtools}  
\usepackage{hyperref}
\hypersetup{
	colorlinks=true,
	linkcolor=blue,
	filecolor=magenta,
	urlcolor=cyan,
	citecolor=blue
}
\mathtoolsset{showonlyrefs} 
\usepackage{authblk}
\usepackage[none]{hyphenat}

\newtheorem{theorem}{Theorem}
\newtheorem{claim}[theorem]{Claim}
\newtheorem{lemma}[theorem]{Lemma}

\newtheorem{definition}{Definition}

\newtheorem{conjecture}[theorem]{Conjecture}

\newcommand{\abs}[1]{\left\lvert{#1}\right\rvert}

\DeclareMathOperator{\ex}{ex}

	\title{
	Bipartite Tur\'an problems for ordered graphs
}
	
	\date{}
	
\author{Abhishek Methuku \thanks{\'Ecole Polytechnique F\'ed\'erale de Lausanne, Switzerland and Discrete Mathematics Group, Institute for Basic Science (IBS), Daejeon, Republic of Korea. Research partially supported by Swiss National Science Foundation grants no. 200020-162884 and 200021-175977, and by IBS-R029-C1. e-mail: \texttt{abhishekmethuku@gmail.com}}
	\qquad Istv\'an Tomon \thanks{\'Ecole Polytechnique F\'ed\'erale de Lausanne, Switzerland. Research partially supported by Swiss National Science Foundation grants no. 200020-162884 and 200021-175977. e-mail: \texttt{istvantomon@gmail.com}}}

\begin{document}



	\maketitle
	
\begin{abstract}
	A zero-one matrix $M$ contains a zero-one matrix $A$ if one can delete some rows and columns of $M$, and turn some 1-entries into 0-entries such that the resulting matrix is $A$.	The extremal number of $A$, denoted by $\ex(n,A)$, is the maximum number of $1$-entries in an $n\times n$ sized matrix $M$ that does not contain $A$. 
	
	A matrix $A$ is column-$t$-partite (or row-$t$-partite), if it can be cut along the columns (or rows) into $t$ submatrices such that every row (or column) of these submatrices contains at most one $1$-entry. We prove that if $A$ is column-$t$-partite, then $\ex(n,A)<n^{2-\frac{1}{t}+\frac{1}{2t^{2}}+o(1)}$, and if $A$ is both column- and row-$t$-partite, then $\ex(n,A)<n^{2-\frac{1}{t}+o(1)}$. Our proof combines a novel density-increment-type argument with the celebrated dependent random choice method.
	
	Results about the extremal numbers of zero-one matrices translate into results about the Tur\'an numbers of bipartite ordered graphs. In particular, a zero-one matrix with at most $t$ 1-entries in each row corresponds to a bipartite ordered graph with maximum degree $t$ in one of its vertex classes. Our results are partially motivated by a well known result of F\"uredi (1991) and Alon, Krivelevich, Sudakov (2003) stating that if $H$ is a bipartite graph with maximum degree $t$ in one of the vertex classes, then $\ex(n,H)=O(n^{2-\frac{1}{t}})$. The aim of the present paper is to establish similar general results about the extremal numbers of ordered graphs.
\end{abstract}
	
\section{Introduction} 

If $H$ is a graph, the \emph{extremal number} (or \emph{Tur\'an number}) of $H$ is
 $$\ex(n,H)=\max \{|E(G)|: |V(G)|=n, G\mbox{ does not contain }H\mbox{ as a subgraph}\}.$$
 By the classical Erd\H{o}s-Stone theorem \cite{ES46}, we have
 $$\ex(n,H)=\left(1-\frac{1}{\chi(H)-1}+o(1)\right)\binom{n}{2}.$$
 Therefore, the asymptotic value of $\ex(n,H)$ is well understood unless $H$ is bipartite. The value of $\ex(n,H)$ for bipartite graphs $H$ is the subject of extensive study. Recently, there were some interesting new results on this topic, see for example \cite{CL18,J18,GJN19}. Below, we discuss a number of well known results concerning the extremal numbers of bipartite graphs, but first, let us introduce the concept of ordered graphs. 
 
 An \emph{ordered graph} is a pair $(G,<)$, where $G$ is graph and $<$ is a total ordering on the vertex set of $G$. If $<$ is clear from the context, we shall write simply $G$ instead of $(G,<)$. The ordered graph $(H,<')$ is an \emph{ordered subgraph} of $(G,<)$ if there exists an order preserving embedding from $V(H)$ to $V(G)$ that preserves edges. The extremal number of the ordered graph $(H,<')$ is 
 $$\ex_{<}(n,(H,<'))=\max \{|E(G)|: |V(G)|=n, (G,<)\mbox{ does not contain }(H,<')\mbox{ as an ordered subgraph}\}.$$
 The \emph{interval chromatic number} of an ordered graph $H$, denoted by $\chi_{<}(H)$, is the minimum number of colors needed to color the vertices of $H$ such that there are no monochromatic edges, and each color class is an interval with respect to the ordering on $V(H)$. The systematic study of ordered graph was initiated by Pach and Tardos in \cite{PT06}, where the following analogue of the Erd\H{o}s-Stone theorem is established:
  $$\ex_{<}(n,H)=\left(1-\frac{1}{\chi_{<}(H)-1}+o(1)\right)\binom{n}{2}.$$
 Therefore, the asymptotic value of $\ex_{<}(n,H)$ is known unless $\chi_{<}(H)=2$. Say that an ordered graph $H$ is \emph{ordered bipartite} if $\chi_{<}(H)=2$. 
 
In case $\chi_{<}(H)=2$, $\ex_{<}(n,H)$ is strongly related to the extremal number of the bi-adjacency matrix of $H$. This extremal number is defined as follows. If $A$ and $M$ are zero-one matrices, we say that $M$ \emph{contains} $A$ if we can delete rows and columns of $M$, and turn $1$-entries of $M$ into $0$-entries such that the resulting matrix is $A$. The \emph{weight} of a zero-one matrix $M$, denoted by $w(M)$, is the number of $1$-entries of $M$. The \emph{extremal number} of the zero-one matrix $A$ is
$$\ex(n,A)=\max \{w(M): M\mbox{ is an }n\times n\mbox{ sized matrix that does not contain }A\}.$$
If $H$ is an ordered bipartite graph with vertex classes $X$ and $Y$, then $H$ naturally corresponds to the zero-one matrix $A_{H}$, whose rows correspond to the elements of $X$, whose columns correspond to the elements of $Y$, and $A_{H}(x,y)=1$ if and only if $xy\in E(H)$. The following connection between the extremal number of $H$ and $A_{H}$ was established by Pach and Tardos \cite{PT06}:
$$\ex(n,A_{H})\leq \ex_{<}(2n,H)= O(\ex(n,A_{H})\log n),$$
and if $\ex(n,A_{H})\geq n^{1+\epsilon}$ for some $\epsilon>0$, then
$$\ex(n,A_{H})=\Theta(\ex_{<}(n,H)).$$
But this tells us that the order of magnitude of $\ex_{<}(n,H)$ and $\ex(n,A_{H})$ are roughly the same, so it is enough to consider one of these extremal numbers. We found that it is easier work with the matrix terminology, so in the rest of our paper, we will use the language of zero-one matrices instead of ordered bipartite graphs.

Say that a matrix is \emph{acyclic} if the corresponding graph is a forest. The extremal number of acyclic matrices is extensively studied, see for example \cite{FH92,KTTW19,P11,PT06,T05}. It was conjectured by F\"uredi and Hajnal \cite{FH92} that if $A$ is a zero-one matrix that corresponds to an ordered forest, then $\ex(n,A)=O(n\log n)$. This was disproved in a weak sense by Pettie \cite{P11}, and the conjecture of F\"uredi and Hajnal is replaced by the following conjecture of Pach and Tardos \cite{PT06}, which is still open: $\ex(n,A)=n(\log n)^{O(1)}$. The most general result connected to this conjecture is due to Kor\'andi, Tardos, Tomon, and Weidert \cite{KTTW19}, who prove that $\ex(n,A)=n^{1+o(1)}$ holds for a large class of acyclic matrices $A$.

However, there are much fewer results about the extremal number of matrices, whose underlying graph is not acyclic. F\"uredi and Hajnal \cite{FH92} conjectured that if $G$ is an ordered bipartite graph and $G'$ is the underlying unordered bipartite graph, then the extremal numbers of $G$ and $G'$ do not differ by much, in particular $\ex_{<}(n,G)=O(\ex(n,G')\log n)$. This was disproved in a strong sense by Pach and Tardos \cite{PT06}: there exist ordered bipartite cycles of arbitrary even length, whose extremal number is at least $\Omega(n^{\frac{4}{3}})$; we discuss this result in more detail in Section \ref{sect:cycles} (where we also prove new bounds for a large class of ordered cycles). On the other hand, by the well known result of Bondy and Simonovits \cite{BS74}, we have $\ex(n,C_{2k})=O(n^{1+\frac{1}{k}})$, where $C_{2k}$ denotes the cycle of length $2k$. More recently, extremal numbers of some ordered cycles were studied by Gy\H{o}ri, Kor\'andi, Methuku, Tomon, Tompkins, and Vizer \cite{GyKMTTV18}.

The aim of the present paper is to provide general upper bounds on the extremal numbers of zero-one matrices that are not acyclic.

A matrix $A$ is \emph{column-$t$-partite} (or \emph{row-$t$-partite}), if it can be cut along the columns (or rows) into $t$ submatrices such that every row (or column) of these matrices contains at most one $1$-entry. Also, a matrix is \emph{$t\times s$-partite} if $A$ is both row-$t$-partite and column-$s$-partite. See Figure \ref{Examplesofrowandcolumnpartite} for examples.

\begin{figure}
	\label{Examplesofrowandcolumnpartite}
\[  
\left(
\begin{array}{cc|cc}
0 & 1 & 0 & 1 \\ 
1 & 0 & 0 & 1 \\ 
1 & 0 & 0 & 1 \\ 
0 & 1 & 1 & 0
\end{array}
\right), 
\left(
\begin{array}{cccc}
0 & 1 & 0 & 0 \\ 
1 & 0 & 1 & 1 \\ \hline
1 & 0 & 1 & 0 \\ 
0 & 1 & 0 & 1
\end{array}
\right),
\left(
\begin{array}{cc|cc}
0 & 1 & 0 & 1 \\ 
1 & 0 & 1 & 0 \\ \hline
1 & 0 & 1 & 0 \\ 
0 & 1 & 0 & 1
\end{array}
\right)
\]
\caption{Examples of a column-$2$-partite (left), row-$2$-partite (middle), and a $2 \times 2$-partite matrix (right).}
\end{figure}

Certain column-1-partite matrices  $A$ are related to Davenport-Schinzel sequences \cite{DS65}, and  satisfy $\ex(n,A)=\Theta(n\alpha(n))$, where $\alpha(n)$ is the extremely slowly growing inverse-Ackermann function \cite{FH92}. On the other hand, it follows from the result of Kor\'andi et al. \cite{KTTW19} that for every column $1$-partite matrix $A$, we have $\ex(n,A)=n^{1+o(1)}$. The $1\times 1$-partite matrices are exactly the permutation matrices, and by the celebrated result of Marcus and Tardos \cite{MT04}, if $A$ is a permutation matrix, then $\ex(n,A)=O(n)$. Column-2-partite matrices correspond to subgraphs of $1$-subdivisions of bipartite multigraphs, see \cite{CL18,J18} for related results in the unordered case, and $2\times 2$-partite matrices correspond to subgraphs of disjoint unions of cycles. 

We prove the following general result about the extremal numbers of column-$t$-partite matrices. (Note that due to symmetry the analogous result for row-$t$-partite matrices also holds.)

\begin{theorem}
	\label{thm:mainthm}
Let $t \ge 2$ be an integer and let $A$ be a column-$t$-partite zero-one matrix. Then $$\ex(n, A) <n^{2 - \frac{1}{t}+\frac{1}{2t^{2}}+o(1)}.$$
\end{theorem}

 As usual, let $K_{s,t}$ denote the complete bipartite graph with vertex classes of sizes $s$ and $t$. By the K\H{o}v\'ari-S\'os-Tur\'an theorem \cite{KST54}, we have $\ex(n,K_{s,t})=O(n^{2-\frac{1}{t}})$, and  by a construction of Alon, R\'onyai, and Szab\'o \cite{ARSz99} if $s\geq (t-1)!$, then $\ex(n,K_{s,t})=\Theta(n^{2-\frac{1}{t}})$.  So if $A$ is the $s\times t$ sized all-1 matrix and $s\geq (t-1)!$, then $A$ is column-$t$-partite and $\ex(n,A)=\Theta(n^{2-\frac{1}{t}})$. This shows that the exponent $2-\frac{1}{t}+\frac{1}{2t^2}+o(1)$ in Theorem \ref{thm:mainthm} cannot be replaced by anything smaller than $2-\frac{1}{t}$. We conjecture that the latter exponent is the truth, i.e., $\ex(n,A)<n^{2-\frac{1}{t}+o(1)}.$ We can prove this conjecture in the case when $A$ is $t\times t$-partite. 

\begin{theorem}
	\label{thm:symmetric}
Let $t \ge 2$ be an integer and let $A$ be a $t\times t$-partite zero-one matrix. Then $$\ex(n, A) <n^{2 - \frac{1}{t}+o(1)}.$$
\end{theorem}

This theorem is sharp up to the $o(1)$ term assuming the following well known conjecture is true: $\ex(n,K_{t,t})=\Theta(n^{2-\frac{1}{t}})$. Let us highlight some of the unordered analogues of Theorem \ref{thm:mainthm} and Theorem \ref{thm:symmetric}. If $A$ is column-$t$-partite, then the corresponding ordered bipartite graph has maximum degree $t$ in one of its vertex classes. A well known result of F\"uredi \cite{F91}, reproved by Alon, Krivelevich, and Sudakov \cite{AKS03} applying the celebrated dependent random choice method, is the following: If $H$ is an (unordered) bipartite graph with maximum degree $t$ in one of its vertex classes, then $\ex(n,H)=O(n^{2-\frac{1}{t}})$. This result served as one of our main motivations for studying the problems above. Bipartite graphs with maximum degree $t$ in one of the vertex classes correspond to zero-one matrices in which every row has at most $t$ $1$-entries. Even though our proof method of Theorem \ref{thm:mainthm} works for a large class of matrices, it does not work for every such matrix. It would be interesting to decide whether Theorem \ref{thm:mainthm} can be extended to matrices with at most $t$ $1$-entries in every row.

Our paper is organized as follows. In Section \ref{sect:thm1}, we prove a somewhat weaker version of Theorem \ref{thm:mainthm}, while introducing our proof method and the main tools used in the paper. Then, we prove Theorem \ref{thm:symmetric} in Section \ref{sect:txt}, and Theorem \ref{thm:mainthm} in Section \ref{sect:mainthm}. In Section \ref{sect:cycles}, we prove that $\ex(n,A)=O(n^{\frac{3}{2}})$ holds for a large number of matrices $A$ that correspond to cycles. We finish our paper with some remarks and open questions in Section \ref{sect:remarks}. 

\subsection{Preliminaries}

We use the following extended definition of the binomial coefficients. If $x\in\mathbb{R}$ and $k\in \mathbb{Z}^{+}$, then
$$\binom{x}{k}=\begin{cases}\frac{x(x-1)\dots(x-k+1)}{k!} & \mbox{ if } x\geq k-1,\\
	0 & \mbox{ otherwise.}\end{cases}$$
We will use the following properties of this binomial coefficient.
\begin{itemize}
	\item For every $k\in\mathbb{Z}^{+}$, the function $f(x)=\binom{x}{k}$ is continuous and convex.
	\item $\binom{x}{k}\leq \frac{x^{k}}{k!}<x^{k}.$
	\item If $x\geq k$, then $\binom{x}{k}\geq (\frac{x}{k})^{k}$.
\end{itemize}

\section{Column-$t$-partite matrices}\label{sect:thm1}

In this section, we prove the following slightly weaker version of Theorem \ref{thm:mainthm}, while outlining the main ideas used in our paper.

\begin{theorem}\label{thm:simplerversion}
Let $t\geq 2$ and let $A$ be a column-$t$-partite matrix. Then
$$\ex(n,A)<n^{2-\frac{1}{t}+\frac{1}{t^2}+o(1)}.$$
\end{theorem}

In the rest of this section, $A$ is an $r\times s$ sized matrix, and we fix a real number $\epsilon>0$. Also, set $k = \left\lceil {(\frac{4 r^{t-1} t^t}{ t!})}^{\frac{1}{\epsilon}} \right\rceil$.  We show that $\ex(n,A)<n^{2-\frac{1}{t}+\frac{1}{t^2}+\epsilon}$ if $n$ is sufficiently large given the parameters $\epsilon,r,s,k$.

A \emph{copy of $K_{u,t}$} in a zero-one matrix $M$ is a submatrix of $M$ induced by $u$ distinct rows and $t$ distinct columns such that every entry in the submatrix is 1. 

\subsection{Overview of the proof}

Let us briefly outline our proof. Let $M$ be an $n\times n$ sized matrix of weight at least $n^{2-\frac{1}{t}+\frac{1}{t^2}+\epsilon}$. We construct a sequence of matrices $M=M_{0},M_{1},\dots$ satisfying the following properties:
\begin{itemize}
	\item $M_{i+1}$ is a submatrix of $M_i$ for every $i$.
	\item The size of $M_{i}$ is $\frac{n}{k^i}\times n$ for every $i$.
	\item With a certain definition of ``density'', the  density of $K_{t,t}$'s in $M_{i}$ is increasing as $i$ increases.
\end{itemize} 
However, if this sequence is long enough, we eventually run into a contradiction as we get matrices which are too small to have the claimed density of $K_{t,t}$'s. 

To construct the above sequence, for each $i$, we divide $M_i$ into $k$ equal sized blocks along its rows, and prove that either $A$ can be embedded into $M_i$ such that each row of $A$ is embedded into a different block (in which case we are done), or one of the blocks must have a high density of $K_{t,t}$'s, and then we set this block to be $M_{i+1}$.

In order to prove this, we adapt the so called dependent random choice method \cite{AKS03}. While this method is usually formulated in a probabilistic way, it can be easily turned into a counting argument, which is more suitable for our purposes. Roughly, the idea is to double count certain ``light'' copies of $K_{t,t}$, where a copy of $K_{t,t}$ is light if it belongs to $t$ columns that only intersect a few of the blocks in $1$-entries.

\subsection{Proof of theorem \ref{thm:simplerversion}}

\begin{definition}
	\label{horizontal_block}
Given an $m \times n$ sized zero-one matrix $M$ and integer $k$, where $k$ divides $m$, the submatrix of $M$ formed by the rows $\frac{(p-1)m}{k}+1, \frac{(p-1)m}{k}+2, \ldots, \frac{pm}{k}$ of $M$ is called a \emph{horizontal block} of $M$ for any integer $p$ with $1 \le p \le k$. Note that this \emph{partitions} $M$ into $k$ horizontal blocks and each horizontal block is of size $\frac{m}{k} \times n$.
\end{definition}
We will need the following two simple lemmas.

\begin{lemma}\label{lemma:counting}
	Let $M$ be an $n\times n$ sized matrix such that $w(M)>tun^{2-\frac{1}{u}}$. Then the number of copies of $K_{u,t}$ in $M$ is at least 
	$$\frac{1}{u^{ut-t+u}t^{t}}\cdot\frac{w(M)^{tu}}{n^{2ut-u-t}}.$$
	
\end{lemma}

\begin{proof}
	Let $N$ be the number of copies of $K_{u,t}$ in $M$. For $1\leq i_{1}<\dots<i_{u}\leq n$, let $n(i_{1},\dots,i_{u})$ be the number of columns $j\in [n]$ such that $M(i_{1},j)=\dots=M(i_{u},j)=1$. Then
	$$N=\sum_{1\leq i_{1}<\dots<i_{u}\leq n}\binom{n(i_{1},\dots,i_{u})}{t}\geq \binom{n}{u}\binom{\frac{1}{\binom{n}{u}}\sum_{1\leq i_{1}<\dots<i_{u}\leq n}n(i_{1},\dots,i_{u})}{t},$$
	where the inequality holds by convexity. But the sum $\sum_{1\leq i_{1}<\dots<i_{u}\leq n}n(i_{1},\dots,i_{u})$ is equal to the sum of the number of $u$ element subsets of $1$-entries contained in a column (with the sum taken over all the columns). Therefore, if $m_{j}$ denotes the number of $1$-entries in the $j$-th column for $j\in [n]$, then
	
	$$\sum_{1\leq i_{1}<\dots<i_{u}\leq n}n(i_{1},\dots,i_{u})=\sum_{j=1}^{n}\binom{m_{j}}{u}\geq n\binom{\frac{1}{n}\sum_{j=1}^{n}m_{j}}{u}=n\binom{\frac{1}{n}w(M)}{u}\geq \frac{w(M)^{u}}{n^{u-1}u^{u}},$$
	where the first inequality holds by convexity, and the second assuming that $w(M)\geq nu$. But then
	$$N\geq\binom{n}{u}\binom{\frac{w(M)^{u}}{n^{u-1}u^{u-1}\binom{n}{u}}}{t}\geq \left(\frac{n}{u}\right)^{u}\binom{\frac{w(M)^{u}}{n^{2u-1}u^{u-1}}}{t}\geq \frac{w(M)^{tu}}{n^{2ut-u-t}u^{ut-t+u}t^{t}},$$
	where the last inequality holds by the assumption $w(M)>tun^{2-\frac{1}{u}}$.
\end{proof}

\begin{lemma}
	\label{lemma:t-partite}
 Let $H$ be a $t$-uniform hypergraph on $[n]$ such that $H$ does not contain a complete $t$-partite hypergraph with parts $V_1, V_2, \ldots, V_t$, where $\abs{V_i} = s$ and every element of $V_i$ is smaller than every element of $V_j$ for every $i < j$. Then $H$ contains at most $2n^{t-\delta}$ hyperedges, where $\delta=\frac{1}{ts^{t-1}}$.
\end{lemma}

\begin{proof}
A \emph{$t$-cut} $(i_1, i_2, \ldots, i_{t-1})$ is a partition of $[n]$ into $t$ parts $\{1,2, \ldots, i_1\}$, $\{i_1+1,  i_1+2, \ldots,i_2 \}$, $\{i_2+1, i_2+2,\ldots, i_3\}, \ldots, \{i_{t-1}+1, i_{t-1}+2, \ldots, n \}$. A \emph{random $t$-cut} $(i_1, i_2, \ldots, i_{t-1})$ is a $t$-cut where the integers $i_1, i_2, \ldots, i_{t-1}$ are chosen uniformly at random with  $i_1 < i_2 < \ldots <i_{t-1}$. An edge $e = \{x_1, x_2, \ldots, x_t\}$ of $H$ with $x_1 < x_2 < \ldots < x_t$ is \emph{cut} by a $t$-cut $(i_1, i_2, \ldots, i_{t-1})$ if $x_j \in \{i_{j-1}+1, i_{j-1}+2,\ldots, i_j\}$ for each $1 \le j \le t$ (where $i_0 = 0$ and $i_t = n$). Then the probability that a given edge $e = \{x_1, x_2, \ldots, x_t\}$ of $H$ with $x_1 < x_2 < \ldots < x_t$ is cut by a random $t$-cut is $(\frac{x_2-x_1}{n}) (\frac{x_3-x_2}{n}) \ldots (\frac{x_t-x_{t-1}}{n})$. We say that $e$ is \emph{bad} if this quantity is less than $n^{-\gamma}$, where $\gamma=\frac{t-1}{ts^{t-1}}$; otherwise $e$ is \emph{good}. Note that if $e$ is bad, then there is a $j = j(e)$ with $1 \le j \le t-1$ such that $\frac{x_{j+1}-x_{j}}{n}$ is less than $n^{-\frac{\gamma}{t-1}}$. So once $x_j$ is fixed, the number of possible choices for $x_{j+1}$ are less than $n^{1-\frac{\gamma}{t-1}}$. Therefore, the number of bad edges $e$ in $H$ is less than $n^{t-\frac{\gamma}{t-1}}$.

Assume that $\abs{E(H)} \ge 2n^{t-\frac{\gamma}{t-1}}$. Then $H$ has at least $n^{t-\frac{\gamma}{t-1}}$ good edges. But by the definition of good edges, the expected number of edges that are cut by a random $t$-cut is at least $n^{t-\frac{\gamma}{t-1}} \cdot n^{-\gamma} = n^{t(1-\frac{\gamma}{t-1})}$. So there exists a $t$-cut that cuts at least $n^{t(1-\frac{\gamma}{t-1})}$ edges.  These $n^{t(1-\frac{\gamma}{t-1})}$ edges form a $t$-partite hypergraph $H'$ whose $t$ classes are the $t$ parts of the $t$-cut. By a well-known result of Erd\H{o}s \cite{E64}, as $|E(H')|\geq  n^{t-\frac{1}{s^{t-1}}}$,  $H'$ contains a complete $t$-partite hypergraph with parts of size $s$. But then $H'$ contains such a $t$-partite hypergraph, whose parts are ordered in the desired manner, leading to a contradiction.  
\end{proof}

Let us make the following remark about Lemma \ref{lemma:t-partite}. Say that a $t$-uniform hypergraph on an ordered vertex set $V$ is interval $t$-partite, if $V$ can be partitioned into $t$ intervals $I_{1},\dots, I_{t}$  (with respect to the ordering) such that each hyperedge of $H$ contains exactly one element of each of the intervals. Recently, it was proved by F\"uredi, Jiang, Kostochka, Mubayi and Verstra\"{e}te \cite{FJKMV19} that for every $t-1<\alpha<t$ there exists $c=c(t,\alpha)$ such that every $t$-uniform hypergraph $H$ on an ordered vertex set of size $n$ with $dn^{\alpha}$ edges contains an interval $t$-partite hypergraph with $m$ vertices and at least $cdm^{\alpha}$ hyperedges for some $m\leq n$. Combining this with the result of Erd\H{o}s \cite{E64} cited in our proof, we get that the bound $2n^{t-\delta}$ in Lemma \ref{lemma:t-partite} can be improved to $O(n^{t-\frac{1}{s^{t-1}}})$. However, this improvement of Lemma \ref{lemma:t-partite} does not change the exponent in Theorem \ref{thm:simplerversion}. We included the simple proof of Lemma \ref{lemma:t-partite} for completeness.

For the rest of the proof, fix $\delta=\frac{1}{ts^{t-1}}$. The following lemma is the heart of the proof.

\begin{lemma}
\label{lemma:densityincrement}
Let $u$ be a positive integer, then there exists a constant $C=C(t,u,r,s,k)$ such that the following holds. Let $m,n$ be positive integers such that $k$ divides $m$, and let $M$ be an $m \times n$ sized matrix which does not contain $A$. Partition $M$ into $k$ horizontal blocks. Let $N$ be the number of copies of $K_{u,t}$ in $M$, and suppose that $N>C\max\{m^{u},n^{t}\}$. Then one of the $k$ horizontal blocks of $M$ contains at least $\frac{u!}{4 r^{u-1} u^u}\frac{N}{k}$ copies of $K_{u,t}$.
\end{lemma}

\begin{proof}
Let us define an edge-labeled $t$-uniform hypergraph $H$ on $[n]$ as follows: $\{j_1, j_2, \ldots, j_t\}$ is an edge of $H$ if there is a row $i$ of $M$ such that $M(i, j_\ell) = 1$ for all $1 \le \ell \le t$, and let us define the function
\begin{multline*}
\phi(\{j_1, j_2, \ldots, j_t\}) = \{ b \in [k] \mid \text{there exists a row $i$ in a horizontal block $b$ } \\ \text{such that $M(i, j_\ell) = 1$ for every $1 \le \ell \le t$} \}.
\end{multline*}
An edge $e$ of $H$ is \emph{light} if $\abs{\phi(e)} < r$, otherwise it is \emph{heavy}. Moreover, if $e$ is heavy, let us label $e$ with an arbitrary subset of $\phi(e)$ of size $r$.

A copy of $K_{u,t}$ is \emph{light} if it is contained in the $t$ columns $j_1, j_2, \ldots, j_t$ such that  $\{j_1, j_2, \ldots, j_t\}$ is a light edge of $H$. Otherwise, it is called a \emph{heavy} copy.  Our goal now is to show that most of the copies of $K_{u,t}$ in $M$ are light. For $u$ distinct rows indexed by $i_1, i_2, \ldots, i_u$, let $N(i_1, i_2, \ldots, i_u)$ be the set of columns $j$ such that $M(i_\ell, j) = 1$ for all $1 \le \ell \le u$.
\begin{claim}
	\label{claim:boundingHeavyCopies}
Let $i_1, i_2, \ldots, i_u$ be the indices of $u$ distinct rows of $M$. Then the number of heavy edges in $H[N(i_1, i_2, \ldots, i_u)]$ is at most $2\binom{k}{r} \abs{N(i_1, i_2, \ldots, i_u)}^{t-\delta}$.
\end{claim}
\begin{proof}
As $A$ is column-$t$-partite, there exists a partition of $[s]$ into $t$ intervals $I_{1},\dots,I_{t}$ such that for every $a\in [r]$ and $c\in [t]$, there exists at most one index $b=b(a,c)\in I_{\ell}$ (for some $\ell$) such that $A(a,b)=1$. Without loss of generality, we can assume that there exists exactly one such index.

   Suppose that there exists a label $L$ and $t$ disjoint sets $V_1, V_2, \ldots, V_t$ in $N(i_{1},\dots,i_{u})$ such that $\abs{V_i} = |I_{i}|$, all the vertices of $V_i$ precede the vertices of $V_j$ for every $i < j$, and for every $(j_{1},\dots,j_{t})\in V_{1}\times\dots \times V_{t}$, $\{j_{1},\dots,j_{t}\}$ is an edge of $H$ with label $L$. In this case, we show that $M$ contains $A$. Indeed, let $l_{1}<\dots<l_{r}$ be the elements of $L$, and let $v_{1}<\dots<v_{s}$ be the elements of $V_{1}\cup\dots\cup V_{t}$. For each $a\in [r]$, there exists a row $i'_{a}$ in block $l_{a}$ such that $M(i'_{a},v_{b(a,1)})=M(i'_{a},v_{b(a,2)})=\dots=M(i'_{a},v_{b(a,t)})=1$. But then the $r\times s$ times submatrix of $M$ indexed by the rows $\{i'_{1},\dots,i'_{r}\}$ and columns $\{v_{1},\dots,v_{s}\}$ contains $A$.

  As $M$ does not contain $A$, Lemma \ref{lemma:t-partite} implies that the number of heavy edges in $H[N(i_1, i_2, \ldots, i_u)]$ with the same label is at most $2\abs{N(i_1, i_2, \ldots, i_u)}^{t-\delta}$. But there are at most $\binom{k}{r}$ possible labels, so the total number of heavy edges in $H[N(i_1, i_2, \ldots, i_t)]$ is at most $\binom{k}{r} \cdot 2\abs{N(i_1, i_2, \ldots, i_t)}^{t-\delta}$, as desired.
\end{proof}

\begin{claim}
	\label{claim:heavycount}
There exists a constant $C'=C'(t,u,r,s,k)$ such that if $N > C'm^{u}$, then the number of heavy copies of $K_{u,t}$ in $M$ is at most $\frac{N}{2}$.
\end{claim}
\begin{proof}
Let $C_{0}$ be a positive real number such that $C_{0}^{\delta}\geq 8t^{t}\binom{k}{r}$. We show that $C'=8\binom{k}{r}C_{0}^{t-\delta}$ suffices. 

Suppose that $N\geq C' m^{u}$. The number of copies of $K_{u,t}$ in $M$ is 
\begin{equation}
\label{eq1}
N = \sum_{1\leq i_1<\ldots<i_u\leq m} \binom{\abs{N(i_1, i_2, \ldots, i_u)}}{t}.
\end{equation}
By Claim \ref{claim:boundingHeavyCopies}, the number of heavy copies of $K_{u,t}$ in $M$ is at most $$\sum_{1\leq i_1<\ldots<i_u\leq m} 2 \binom{k}{r} \abs{N(i_1, i_2, \ldots, i_t)}^{t-\delta}.$$
 For integers $1\leq i_1<\dots<i_u\leq m$, we say that $(i_1, i_2, \ldots, i_u)$ is \emph{bad} if $N(i_1, i_2, \ldots, i_u) < C_{0}$, otherwise $(i_1, i_2, \ldots, i_u)$ is \emph{good}. Then

\begin{equation}
\label{eq:bad}
\sum_{(i_1, i_2, \ldots, i_u) \text{ is bad}}  2 \binom{k}{r} \abs{N(i_1, i_2, \ldots, i_u)}^{t-\delta} < m^u  2 \binom{k}{r} C_{0}^{t-\delta} \le \frac{N}{4}.
\end{equation}
Moreover, if $(i_1, i_2, \ldots, i_u)$ is \emph{good}, then 

$$\abs{N(i_1, i_2, \ldots, i_u)}^{t-\delta} \le \frac{\abs{N(i_1, i_2, \ldots, i_u)}^{t}}{C_{0}^{\delta}} \le \frac{t^{t}}{C_{0}^{\delta}} \binom{\abs{N(i_1, i_2, \ldots, i_u)}}{t}\leq \frac{1}{8\binom{k}{r}}\binom{\abs{N(i_1, i_2, \ldots, i_u)}}{t}.$$
Therefore, using \eqref{eq1}, we get
\begin{multline}
\label{eq:good}
\sum_{(i_1, i_2, \ldots, i_u) \text{ is good}} 2 \binom{k}{r} \abs{N(i_1, i_2, \ldots, i_u)}^{t-\delta}\leq \sum_{1\leq i_1<\dots<i_u\leq m} \frac{1}{4}\binom{\abs{N(i_1, i_2, \ldots, i_u)}}{t} =\frac{N}{4}.
\end{multline}
Combining \eqref{eq:bad} and \eqref{eq:good} finishes the proof of the claim.
\end{proof}

We show that the constant $C=\max\{C',4\binom{ru}{u}\}$ satisfies the desired properties of Lemma \ref{lemma:densityincrement}, where $C'$ is the constant defined in Claim \ref{claim:heavycount}. Suppose that $N\geq Cm^{u}$, then by Claim \ref{claim:heavycount}, the number of light copies of $K_{u,t}$ in $M$ is at least $\frac{N}{2}$.  For distinct integers $j_1, j_2, \ldots, j_t \in [n]$, let $N'(j_1, j_2, \ldots, j_t)$ be the set of rows $i$ such that $M(i, j_\ell) = 1$ for all $1 \le \ell \le t$. Then the number of light copies of $K_{u,t}$ in $M$ is $$\sum_{\{j_1, j_2, \ldots, j_t\} \text{ is light} } \binom{\abs{N'(j_1, j_2, \ldots, j_t)}}{u}.$$ Thus, 

\begin{equation}
\sum_{\{j_1, j_2, \ldots, j_t\} \text{ is light} } \binom{\abs{N'(j_1, j_2, \ldots, j_t)}}{u} \ge \frac{N}{2}.
\end{equation}

A copy of $K_{u,t}$ is called \emph{narrow} if it is completely contained within a horizontal block. If $\{j_1, j_2, \ldots, j_t\} $ is a light edge of $H$, then the elements of $N'(j_1, j_2, \ldots, j_t)$ are contained in less than $r$ blocks, say, $b_1, b_2, \ldots, b_{\ell}$ for some $\ell < r$. If $N'_q(j_1, j_2, \ldots, j_t)$ denotes the number of elements of $N'(j_1, j_2, \ldots, j_t)$ that are contained in the horizontal block $b_q$, then the number of narrow copies of $K_{u,t}$ in $M$ that are contained in the columns $j_1, j_2, \ldots, j_t$ is at least $$\sum_{1 \le q \le \ell} \binom{\abs{N'_q(j_1, j_2, \ldots, j_t)}}{u} \ge
\ell \binom{\frac{1}{\ell}\abs{N'(j_1, j_2, \ldots, j_t)}}{u},$$ 
where the first inequality holds by the convexity of the function $f(x)=\binom{x}{u}$. Say that $\{j_{1},\dots,j_{t}\}$ is \emph{bad} if $|N(j_{1},\dots,j_{t})|\leq ru$, otherwise $\{j_{1},\dots,j_{t}\}$ is good. We have
$$\sum_{\{j_{1},\dots,j_{t}\}\text{ is bad}}\binom{\abs{N'(j_1, j_2, \ldots, j_t)}}{u}\leq \binom{ru}{u}n^{t}\leq \frac{C n^t}{4} < \frac{N}{4},$$
which implies
\begin{equation}
\label{eq:lightcount2}
\sum_{\substack{\{j_1, j_2, \ldots, j_t\} \text{ is}\\\text{light and good}} } \binom{\abs{N'(j_1, j_2, \ldots, j_t)}}{u} \ge \frac{N}{4}.
\end{equation}
On the other hand, if $\{j_{1},\dots,j_{t}\}$ is good, then 
$$\ell \binom{\frac{1}{\ell}\abs{N'(j_1, j_2, \ldots, j_t)}}{u}\geq \frac{\abs{N'(j_1, j_2, \ldots, j_t)}^u}{r^{u-1} u^u}.$$
Therefore, using \eqref{eq:lightcount2}, the total number of narrow copies of $K_{u,t}$ in $M$ is at least $$\sum_{\substack{\{j_1, j_2, \ldots, j_t\} \text{ is} \\ \text{light and good}}}\frac{\abs{N'(j_1, j_2, \ldots, j_t)}^u}{r^{u-1} u^u} \ge \frac{N u! }{4 r^{u-1} u^u}.$$

As there are $k$ horizontal blocks in total, and by definition, each narrow copy of $K_{u,t}$ completely lives within one of these blocks, there is a horizontal block of $M$ which contains at least $\frac{N u!}{4 r^{u-1} u^u k}$ copies of $K_{u,t}$, completing the proof of Lemma  \ref{lemma:densityincrement}.
\end{proof}

Now we are ready to prove Theorem \ref{thm:simplerversion}. Suppose that there exists arbitrarily large $n$ such that $\ex(n,A)\geq n^{2-\frac{1}{t}+\frac{1}{t^{2}}+\epsilon}$. But then there exists arbitrarily large $n$ such that $n$ is a power of $k$ and $\ex(n,A)\geq n^{2-\frac{1}{t}+\frac{1}{t^{2}}+\frac{\epsilon}{2}}$. Indeed, if $n$ is such that $\ex(n,A)\geq n^{2-\frac{1}{t}+\frac{1}{t^{2}}+\epsilon}$ and $n'$ is the smallest power of $k$ such that $n'\geq n$, then 
$$\ex(n',A)\geq \ex(n,A)\geq n^{2-\frac{1}{t}+\frac{1}{t^{2}}+\epsilon}\geq \left(\frac{n'}{k}\right)^{2-\frac{1}{t}+\frac{1}{t^{2}}+\epsilon}\geq (n')^{2-\frac{1}{t}+\frac{1}{t^{2}}+\frac{\epsilon}{2}},$$
assuming $n$ is sufficiently large, given $k, t$ and $\epsilon$.

 Therefore, we can suppose that $n=k^{z}$, where $z$ is some integer and $\ex(n,A)\geq n^{2-\frac{1}{t}+\frac{1}{t^{2}}+\frac{\epsilon}{2}}$. Let $M$ be an $n \times n$ matrix that does not contain $A$ such that $w(M)=\ex(n,A)$, and let $N$ denote the number of copies of $K_{t,t}$ in $M$. Also, let $C=C(t,t,r,s,k)$ be the constant defined in Lemma \ref{lemma:densityincrement}. By Lemma \ref{lemma:counting},
\begin{equation}
\label{eq:supersaturation}
N \geq c\frac{w(M)^{t^{2}}}{n^{2t^{2}-2t}}\geq c n^{t+1+2\epsilon},
\end{equation}
where $c=t^{-t^{2}-t}$. We construct a sequence of matrices $M=M_{0},M_{1},\dots,M_{z}$ such that for $i=0,1,\dots,z$, the size of $M_{i}$ is $\frac{n}{k^{i}}\times n$, and $M_{i}$ contains $N_{i}$ copies of $K_{t,t}$, where $N_{i}\geq \frac{N}{k^{(1+\epsilon)i}}$. If $M_{i}$ is already defined satisfying these properties for some $i<z$, we define $M_{i+1}$ as follows. First, note that if $n\geq (\frac{C}{c})^{\frac{1}{\epsilon}}$, we have $N_{i}\geq Cn^{t}$. Indeed, remembering that $k^{i}< n$, we can write
	$$N_{i}\geq \frac{N}{k^{(1+\epsilon)i}}\geq \frac{c n^{t+1+2\epsilon}}{k^{(1+\epsilon)i}}\geq c n^{t+\epsilon}\geq Cn^{t}.$$
Hence, we can apply Lemma \ref{lemma:densityincrement} to find an $\frac{n}{k^{i+1}}\times n$ sized submatrix $M_{i+1}$ of $M_{i}$ with the following property: if $N_{i+1}$ denotes the number of copies of $K_{t,t}$ in $M_{i+1}$, then $N_{i+1}\geq \frac{N_{i} t!}{4r^{t-1}t^{t}k}$. But we chose $k$ such that $k \ge {(\frac{4 r^{t-1} t^t}{t!})}^{\frac{1}{\epsilon}}$, so we have $N_{i+1}\geq \frac{N_{i}}{k^{1+\epsilon}}\geq \frac{N}{k^{(1+\epsilon)(i+1)}}$, satisfying the desired properties.

But then we arrive to a contradiction: the size of the matrix  $M_{z}$ is $1\times n$, but $M_{z}$ contains at least $N_{z}\geq \frac{N}{k^{z(1+\epsilon)}}\geq c n^{t+\epsilon}$ copies of $K_{t,t}$, which is clearly impossible. In particular, we run into a contradiction much earlier: the matrix $M_{\lceil (1-\epsilon/t)z\rceil}$ has less than $n^{\frac{\epsilon}{t}}$ rows, but it  contains more than $n^{t+\epsilon}$ copies of $K_{t,t}$, which is also impossible.

\section{$t\times t$-partite matrices - Proof of Theorem \ref{thm:symmetric}}\label{sect:txt}

In the rest of this section, $A$ is a $t\times t$-partite $r \times s$ sized zero-one matrix.  Fix a real number $\epsilon > 0$, and set $k =\left \lceil \left(\frac{16 (rs)^{t-1} t^{2t}}{(t!)^{2}}\right)^{\frac{1}{\epsilon}} \right \rceil$. We prove that $\ex(n,A)\leq n^{2-\frac{1}{t}+\epsilon}$ if $n$ is sufficiently large with respect to the parameters $\epsilon,r,s,k$.

 In addition to horizontal blocks (Definition \ref{horizontal_block}), we introduce the notion of vertical blocks and that of blocks.

\begin{definition}
	Given an $m \times n$ sized zero-one matrix $M$ and integer $k$, where $n$ is divisible by $k$, the submatrix of $M$ formed by the columns $\frac{(p-1)m}{k}+1, \frac{(p-1)m}{k}+2, \ldots, \frac{pm}{k}$ of $M$ is called a \emph{vertical block} of $M$ for any integer $p$ with $1 \le p \le k$. Note that vertical blocks partition $M$ into $k$ submatrices of size $m \times \frac{n}{k}$.
	
	Also, if both $m$ and $n$ are divisible by $k$, a \emph{block} is the intersection of a vertical and a horizontal block. Note that blocks partition $M$ into $k^{2}$ submatrices of size $\frac{n}{k}\times\frac{m}{k}$.
\end{definition}

We prove the following extension of Lemma \ref{lemma:densityincrement} for matrices not containing $A$.

\begin{lemma}
	\label{lemma:densityincrementsymmetric}
There exists a constant $C=C(t,r,s,k)$ such that the following holds. Let $n$ be a positive integer such that $k$ divides $n$, and let $M$ be an $n \times n$ sized matrix which does not contain $A$. Partition $M$ into $k^{2}$ blocks. Let $N$ be the number of copies of $K_{t,t}$ in $M$, and suppose that $N>Cn^{t}$. Then one of the $k^{2}$ blocks of $M$ contains at least $\frac{(t!)^{2}}{16 (rs)^{t-1} t^{2t}}\frac{N}{k^{2}}$ copies of $K_{t,t}$.
\end{lemma}

\begin{proof}
	Let $C'=\max\{C(t,t,r,s,k),C(t,t,s,r,k)\}$, where $C(u,t,r,s,k)$ is the constant given by Lemma \ref{lemma:densityincrement}. We show that $C=\frac{4s^{t-1}t^{t}k}{t!}C'$ suffices. As $A$ is column-$t$-partite and $N\geq C'n^{t}$, we can apply Lemma \ref{lemma:densityincrement} to find a horizontal block $M'$ of $M$ with at least $N'=\frac{t!}{4 r^{t-1} t^t}\cdot \frac{N}{k}$ copies of $K_{t,t}$. Consider the partition of $M'$ into $k$ vertical blocks. As $A$ is also row-$t$-partite and $N'\geq C'n^{t}$, we can apply the symmetric version of Lemma \ref{lemma:densityincrement} to find a horizontal block $M''$ of $M'$ with at least  $\frac{t!}{4 s^{t-1} t^t}\cdot \frac{N'}{k^{2}}=\frac{(t!)^{2}}{16 (rs)^{t-1} t^{2t}}\frac{N}{k^{2}}$ copies of $K_{t,t}$. But $M''$ is a block of $M$, so we are done.
\end{proof}

Using this lemma, we can just repeat the proof of Theorem \ref{thm:simplerversion} with some modifications.

Suppose that there exists arbitrarily large $n$ such that $\ex(n,A)\geq n^{2-\frac{1}{t}+\epsilon}$. By the same argument as before, we can assume that there exists arbitrarily large $n$ such that $n$ is a power of $k$ and $\ex(n,A)\geq n^{2-\frac{1}{t}+\frac{\epsilon}{2}}$.

Hence, let $n=k^{z}$, where $z$ is some integer such that $\ex(n,A)\geq n^{2-\frac{1}{t}+\frac{\epsilon}{2}}$. Let $M$ be an $n \times n$ matrix that does not contain $A$ such that $w(M)=\ex(n,A)$, and let $N$ denote the number of copies of $K_{t,t}$ in $M$. Also, let $C=C(t,r,s,k)$ be the constant defined in Lemma \ref{lemma:densityincrementsymmetric}. By Lemma \ref{lemma:counting},
\begin{equation}
\label{eq:supersaturation2}
N \geq c\frac{w(M)^{t^{2}}}{n^{2t^{2}-2t}}\geq c n^{t+2\epsilon},
\end{equation}
where $c=t^{-t^{2}-t}$. We construct a sequence of matrices $M=M_{0},M_{1},\dots,M_{z}$ such that for $i=0,1,\dots,z$, the size of $M_{i}$ is $\frac{n}{k^{i}}\times \frac{n}{k^{i}}$, and $M_{i}$ contains $N_{i}$ copies of $K_{t,t}$, where $N_{i}\geq \frac{N}{k^{(2+\epsilon)i}}$. If $M_{i}$ is already defined satisfying these properties for some $i<z$, we define $M_{i+1}$ as follows. First, note that if $n\geq (\frac{C}{c})^{\frac{1}{\epsilon}}$, we have $N_{i}\geq C(\frac{n}{k^{i}})^{t}$. Indeed, remembering that $k^{i}< n$, we can write
$$N_{i}\geq \frac{N}{k^{(2+\epsilon)i}}\geq \frac{c n^{t+2\epsilon}}{k^{(2+\epsilon)i}}=c\left(\frac{n}{k^{i}}\right)^{t}n^{2\epsilon}k^{(t-2-\epsilon)i}\geq c n^{\epsilon}\left(\frac{n}{k^{i}}\right)^{t}\geq C\left(\frac{n}{k^{i}}\right)^{t}.$$
Hence, we can apply Lemma \ref{lemma:densityincrementsymmetric} to find an $\frac{n}{k^{i+1}}\times \frac{n}{k^{i+1}}$ sized submatrix $M_{i+1}$ of $M_{i}$ with the following property: if $N_{i+1}$ denotes the number of copies of $K_{t,t}$ in $M_{i+1}$, then $N_{i+1}\geq \frac{(t!)^{2}}{16 (rs)^{t-1} t^{2t}}\cdot\frac{N_{i}}{k^{2}}$. But we chose $k$ such that $k^{\epsilon} \ge \frac{16 (rs)^{t-1} t^{2t}}{(t!)^{2}}$, so we have $N_{i+1}\geq \frac{N_{i}}{k^{2+\epsilon}}\geq \frac{N}{k^{(2+\epsilon)(i+1)}}$, satisfying the desired properties.

But then we arrive to a contradiction: the size of the matrix  $M_{z}$ is $1\times 1$, but $M_{z}$ contains at least $N_{z}\geq \frac{N}{k^{z(2+\epsilon)}}\geq c n^{\epsilon}$ copies of $K_{t,t}$, which is clearly impossible.

\section{Column-$t$-partite matrices revisited -- Proof of Theorem \ref{thm:mainthm}}\label{sect:mainthm}

Similarly as in the proof of Theorem \ref{thm:simplerversion}, given an $n\times n$ sized matrix $M$ of large weight not containing $A$, we define a sequence of matrices $M=M_{0},M_{1},\dots$ such that the size of $M_{i}$ is $\frac{n}{k^{i}}\times n$ with some constant $k$. However, instead of counting copies of $K_{t,t}$ in $M_{i}$, we will count copies of $K_{u,t}$, where $u$ depends on $i$. Indeed, choosing the largest $u$ such that the number of copies of $K_{u,t}$ in $M_{i}$ is $\Omega((\frac{n}{k^{i}})^{u})$, we can apply Lemma \ref{lemma:densityincrement}, and it leads to an improvement over Theorem \ref{thm:simplerversion}.

With the help of the following lemma, we can relate the number of copies of $K_{u,t}$ and $K_{u+1,t}$.

\begin{lemma}\label{lemma:steppingup}
	Let $u$ be a positive integer. Let $M$ be an $m\times n$ size matrix and let $N$ be the number of copies of $K_{u,t}$ in $M$. If $N\geq 2\binom{n}{t}$, then the number of copies of $K_{u+1,t}$ is at least $\frac{1}{2}N^{\frac{u+1}{u}}n^{-\frac{t}{u}}$.
\end{lemma}

\begin{proof}
	For $1\leq j_{1}<\dots<j_{t}\leq n$, let $n(j_{1},\dots,j_{t})$ be the number of rows $i\in[m]$ such that $M(i,j_{1})=\dots=M(i,j_{t})=1$. Then 
	$$N=\sum_{1\leq j_{1}<\dots<j_{t}\leq n}\binom{n(j_{1},\dots,j_{t})}{u}.$$
	Let $a$ be an arbitrary positive integer. If $a\geq u+1$, then $a> \binom{a}{u}^{\frac{1}{u}}$ and
	$$\binom{a}{u+1}=\binom{a}{u}\frac{a-u}{u+1}\geq \binom{a}{u}\frac{a}{(u+1)^{2}}>\frac{1}{(u+1)^{2}}\binom{a}{u}^{\frac{u+1}{u}}.$$
	Also, if $a\leq u$, then $\binom{a}{u+1}>\frac{1}{(u+1)^{2}}\binom{a}{u}^{\frac{u+1}{u}}-1$. But then the number of copies of $K_{u+1,t}$ in $M$ is	
	$$\sum_{1\leq j_{1}<\dots<j_{t}\leq n}\binom{n(j_{1},\dots,j_{t})}{u+1}>\sum_{1\leq j_{1}<\dots<j_{t}\leq n}\binom{n(j_{1},\dots,j_{t})}{u}^{\frac{u+1}{u}}-1.$$
	Using the convexity of the function $f(x)=x^{\frac{u+1}{u}}$, the left hand side is at least
	$$\binom{n}{t}\left(\frac{1}{\binom{n}{t}}\sum_{1\leq j_{1}<\dots<j_{t}\leq n}\binom{n(j_{1},\dots,j_{t})}{u}\right)^{\frac{u+1}{u}}-\binom{n}{t}=\binom{n}{t}^{-\frac{1}{u}}N^{\frac{u+1}{u}}-\binom{n}{t}.$$
	Using that $N\geq 2\binom{n}{t}$, we can write	
	$$\binom{n}{t}^{-\frac{1}{u}}N^{\frac{u+1}{u}}-\binom{n}{t}>\frac{1}{2}\binom{n}{t}^{-\frac{1}{u}}N^{\frac{u+1}{u}}>\frac{1}{2}n^{-\frac{t}{u}}N^{\frac{u+1}{u}}.$$
\end{proof}

In the rest of this section, $r\times s$ is the size of $A$, and $\epsilon>0$ is a fixed real number. Also, fix the following parameters: let $\epsilon_{0}=\frac{\epsilon}{10t^{2}}$, $U=\lceil \frac{10t}{\epsilon_{0}}\rceil$, $\delta=\frac{\epsilon}{10U}$, $k=\lceil(\frac{U!}{4r^{U-1}U^{U-1}})^{\frac{1}{\delta}}\rceil$, and $$C=\max_{u: t\leq u\leq U}C(t,u,r,s,k),$$ where $C(t,u,r,s,k)$ is the constant defined in Lemma \ref{lemma:densityincrement}. Finally, if we say $n$ is sufficiently large, we mean that $n$ is larger than some function of the previously described parameters.

We prove that if $n$ is sufficiently large, then $\ex(n,A)\leq n^{2-\frac{1}{t}+\frac{1}{2t^{2}}+\epsilon}$. Suppose that there exists infinitely many $n$ such that $\ex(n,A)> n^{2-\frac{1}{t}+\frac{1}{2t^{2}}+\epsilon}$. Then, by a similar argument as before, there are infinitely many positive integers $n$ such that $n$ is a power of $k^{U!}$ and $\ex(n,A)> n^{2-\frac{1}{t}+\frac{1}{2t^{2}}+\frac{\epsilon}{2}}$.

Therefore, let $n=k^{z}$, where $z$ is some integer divisible by $U!$ such that $\ex(n,A)\geq n^{2-\frac{1}{t}+\frac{1}{2t^{2}}+\frac{\epsilon}{2}}$. Let $M$ be an $n \times n$ matrix that does not contain $A$ such that $w(M)=\ex(n,A)$.

 Define the decreasing sequence of positive real numbers $\lambda_{t},\lambda_{t+1},\dots,\lambda_{U},\lambda_{U+1}$ as follows. Let $\lambda_{t}=1$, $\lambda_{t+1}=1-\frac{1}{2(t+1)}+\epsilon_{0}$, $\lambda_{U+1}=0$, and set $$\lambda_{u+1}=\lambda_{u}-\frac{t}{(u-1)(u+1)}$$ for $u=t+1,\dots,U-1$. Using the identity $\frac{t}{(u-1)(u+1)}=\frac{t}{2(u-1)}-\frac{t}{2(u+1)},$ one can easily calculate that $$\lambda_{u}=\lambda_{t+1}-\frac{t}{2t}-\frac{t}{2(t+1)}+\frac{t}{2(u-1)}+\frac{t}{2u}=\epsilon_{0}+\frac{t}{2(u-1)}+\frac{t}{2u}$$
 for $u=t+1,\dots,U$. Say that a positive integer $i$ is a \emph{jump}, if there exists $u$ such that $t+1\leq u\leq U$ and $i=z-z\lambda_{u}$. Note that as $z$ is divisible by $U!$, $z\lambda_{u}$ is an integer, so every $z-z\lambda_{u}$ is a jump. Also, if $0\leq i\leq z$, say that $i$ is \emph{type-$u$} if $z-z\lambda_{u}\leq i\leq z-z\lambda_{u+1}$. Note that if $i$ is not a jump, then $i$ has a unique type, but if $i$ is a jump and $i=z-z\lambda_{u}$, then $i$ is both type-$u$ and type-$(u-1)$.

 Let $N$ be the number of copies of $K_{t,t}$ in $M$. We construct a sequence of matrices $M=M_{0},M_{1},\dots,M_{z}$ with the following properties:
 \begin{enumerate}
 	\item $M_{i}$ is a submatrix of $M$ of size $\frac{n}{k^{i}}\times n$.
 	\item Let $N_{u,i}$ denote the number of copies of $K_{u,t}$ in $M_{i}$. If $t+1\leq u\leq U$ and $i$ is type-$u$, then  
 	$$N_{u,i}\geq \frac{n^{u\lambda_{u}+\epsilon}}{k^{(1+\delta)(i-z+z\lambda_{u})}},$$
 	and if $i$ is type-$t$, then
 	$$N_{t,i}\geq \frac{N}{k^{(1+\delta)i}}.$$
 \end{enumerate} 
 
 We construct the sequence $M=M_{0},M_{1},\dots, M_{z}$ by recursion on $i$. In the base case $i=0$, we can apply Lemma \ref{lemma:counting} to get 
 $$N_{t,0}=N\geq c \frac{w(M)^{t^{2}}}{n^{2t^{2}-2t}}\geq c n^{t+\frac{1}{2}+2\epsilon},$$
 where $c=t^{-t^{2}-t}$, so $M_{0}$ satisfies the desired properties. 
 
 Now suppose that we constructed $M_{i}$ with the desired properties for $0\leq i<z$, then we construct $M_{i+1}$ as follows. Let $u$ be the unique integer such that $t\leq u\leq U$ and $z\lambda_{u+1}< z-i \leq z\lambda_{u}$. We would like to apply  Lemma \ref{lemma:densityincrement} to $M_{i}$ to find an $\frac{n}{k^{i+1}}\times n$ sized submatrix of $M_{i}$ with many copies of $K_{u,t}$, but for this, we need to verify that the conditions of Lemma \ref{lemma:densityincrement} are satisfied. That is, we need to show that:
 
 \begin{claim}\label{claim:analysis}
 $N_{u,i}\geq C\max\{(\frac{n}{k^{i}})^{u},n^{t}\}$.
\end{claim}
\begin{proof}
 First, we show that $N(u,i)\geq Cn^{t}$. Here, we consider three cases.
 
 \begin{description}
 	\item[Case 1.] $u=t$.
 	We have
 	$$N_{t,i}\geq \frac{N}{k^{(1+\delta)i}}\geq c \frac{n^{t+\frac{1}{2}+2\epsilon}}{k^{i(1+\delta)}}.$$
 	Here, $i\leq z-z\lambda_{t+1}=z(\frac{1}{2(t+1)}-\epsilon_{0})$. Therefore, the right hand side is at least
 	$$ c n^{t+\frac{1}{2}+2\epsilon-(1+\delta)(\frac{1}{2(t+1)}-\epsilon_{0})}>c n^{t+\frac{1}{2}-\frac{1}{t+1}}\geq c n^{t+\frac{1}{6}}.$$
 	Hence, if $n$ is sufficiently large, then $N_{t,i}\geq Cn^{t}$.
 	
 	\item[Case 2.]$t+1\leq u\leq U-1$.
 	
 	Note that $i-z+z\lambda_{u}<z(\lambda_{u}-\lambda_{u+1})=z\frac{t}{(u-1)(u+1)}$.  Hence,
 	$$N_{u,i}\geq \frac{n^{u\lambda_{u}}}{k^{(1+\delta)(i-z+z\lambda_{u})}}\geq n^{u\lambda_{u}-(1+\delta)\frac{t}{(u-1)(u+1)}}.$$
 	Here, using that $\lambda_{u}\geq \frac{t}{2(u-1)}+\frac{t}{2u}$, we have $$u\lambda_{u}-\frac{(1+\delta)t}{(u-1)(u+1)}\geq\frac{ut}{2(u-1)}+\frac{t}{2}-\frac{(1+\delta)t}{(u-1)(u+1)}= t+\frac{t}{2(u-1)}\left(1-\frac{2(1+\delta)}{u+1}\right)\geq t+\frac{t}{10U}.$$
 	Therefore, $N_{u,i}\geq n^{t+\frac{t}{10U}}$, so if $n$ is sufficiently large, we have $N_{u,i}\geq Cn^{t}$.
 	
 	\item[Case 3.] $u=U$.
 	
 	Here, we have $i-z+z\lambda_{U}\leq z\lambda_{U}$, so
 	$$N_{U,i}\geq \frac{n^{U\lambda_{U}}}{k^{(1+\delta)(i-z+z\lambda_{U})}}\geq n^{(U-1-\delta)\lambda_{U}}.$$
 	Again, using that $\lambda_{U}=\epsilon_{0}+\frac{t}{2(U-1)}+\frac{t}{2U}$, we can write
 	$$(U-1-\delta)\lambda_{U}>(U-1)\lambda_{U}-\delta=t-\frac{t}{2U}+\epsilon_{0}(U-1)-\delta>t+\epsilon_{0}.$$
 	Hence, $N_{U,i}\geq n^{t+\epsilon_{0}}>Cn^{t}$ if $n$ is sufficiently large.	
 	
 \end{description}

 Now, we show that $N_{u,i}\geq C(\frac{n}{k^{i}})^{u}$. Consider two cases.
 	
 \begin{description}
 	\item[Case 1.]  $u=t$.
 	
 	We have already proved that $N_{t,i}\geq Cn^{t}$, so $N_{t,i}\geq C(\frac{n}{k^{i}})^{t}$ follows immediately.
 	
 	\item[Case 2.] $t+1\leq u\leq U$.
 	
 	We have
 	$$N_{u,i}\geq \frac{n^{u\lambda_{u}+\epsilon}}{k^{(1+\delta)(i-z+z\lambda_{u})}}=\left(\frac{n}{k^{i}}\right)^{u}n^{\epsilon}\left(\frac{k^{i}}{n^{(1-\lambda_{u})}}\right)^{u-1-\delta}\geq \left(\frac{n}{k^{i}}\right)^{u}n^{\epsilon}.$$
 	Hence, if $n$ is sufficiently large, we have $N_{u,i}\geq C(\frac{n}{k^{i}})^{u}$.
 \end{description}

 \end{proof}

 Therefore, we can apply Lemma \ref{lemma:densityincrement} to find an $\frac{n}{k^{i+1}}\times n$ sized submatrix $M_{i+1}$ of $M_{i}$ such that the number of copies of $K_{u,t}$ in $M_{i+1}$ is at least
 $$N_{u,i+1}\frac{u!}{4r^{u-1}u^{u-1}}\cdot\frac{N_{u,i}}{k}\geq \frac{N_{u,i}}{k^{1+\delta}}\geq \begin{cases} n^{u\lambda_{u}+\epsilon}k^{-(1+\delta)(i+1-z+z\lambda_{u})} &\mbox{ if }t+1\leq u\leq U\\
 Nk^{-(1+\delta)(i+1)} &\mbox{ if }u=t\end{cases},$$
 where the first inequality holds by the choice of $k$, and the second inequality holds by the induction hypothesis. Therefore $M_{i+1}$ satisfies the desired properties if $i+1$ is not a jump. However, if $i+1$ is a jump, that is $i+1=z-z\lambda_{u+1}$, we also have to verify that for the number of copies of $K_{u+1,t}$ in $M_{i+1}$ we have $N_{u+1,i+1}\geq n^{(u+1)\lambda_{u+1}+\epsilon}$. But this follows from Lemma \ref{lemma:steppingup}. Indeed, by Claim \ref{claim:analysis}, we have $N_{u,i+1}>Cn^{t}>2\binom{n}{t}$, so we can apply Lemma \ref{lemma:steppingup}. Consider two cases. 
 
 \begin{description}
 	\item[Case 1.] $t+1\leq u\leq U$.
 	
 	We have $$N_{u,i+1}\geq \frac{n^{u\lambda_{u}+\epsilon}}{k^{(1+\delta)(i+1-z+z\lambda_{u})}}=\frac{n^{u\lambda_{u}+\epsilon}}{k^{(1+\delta)z(\lambda_{u}-\lambda_{u+1})}}=n^{u\lambda_{u}+\epsilon-(1+\delta)(\lambda_{u}-\lambda_{u+1})},$$
 	where the exponent is
 	\begin{align*}
 	u\lambda_{u}+\epsilon-(1+\delta)(\lambda_{u}-\lambda_{u+1})&>(u-1)\lambda_{u}+\lambda_{u+1}+\epsilon-\delta\\
 	&=u\lambda_{u+1}-\frac{t}{u+1}+\epsilon-\delta\\
 	&>u\lambda_{u+1}-\frac{t}{u+1}+\left(1-\frac{1}{2(u+1)}\right)\epsilon
 	\end{align*}
 	Therefore, by Lemma \ref{lemma:steppingup}, we get
 	$$N_{u+1,i+1}\geq \frac{1}{2}N_{u,i+1}^{\frac{u+1}{u}}n^{-\frac{t}{u}}>\frac{1}{2}n^{(u+1)\lambda_{u+1}+\frac{2u+1}{2u}\epsilon}.$$
 	Hence, if $n>2^{\frac{2u}{\epsilon}}$, we get $N_{u+1,i+1}\geq n^{(u+1)\lambda_{u+1}+\epsilon}$, so $M_{i+1}$ truly satisfies the desired properties.
 	
 	\item[Case 2.] $u=t$.
 	
 	We have $$N_{t,i+1}\geq \frac{N}{k^{(1+\delta)(i+1)}}=\frac{N}{n^{(1+\delta)(1-\lambda_{t+1})}}\geq c n^{t+\frac{1}{2}+2\epsilon-(1+\delta)(1-\lambda_{t+1})},$$
 	where the exponent is
 
 	$$t+\frac{1}{2}+2\epsilon-(1+\delta)(1-\lambda_{t+1})>t+\frac{1}{2}+2\epsilon-\frac{1}{2(t+1)}+\epsilon_{0}-\delta>t+\frac{1}{2}+2\epsilon-\frac{1}{2(t+1)}.$$

 	Therefore, by Lemma \ref{lemma:steppingup}, we get
 	$$N_{t+1,i+1}\geq \frac{1}{2}N_{t,i+1}^{\frac{t+1}{t}}n^{-1}>\frac{1}{2}c^{\frac{t+1}{t}}n^{\frac{2t+1}{2}+\frac{2(t+1)}{t}\epsilon}>\frac{1}{2}c^{\frac{t+1}{t}}n^{(t+1)\lambda_{t+1}+2\epsilon},$$
 	where the last inequality holds noting that $(t+1)\epsilon_{0}<\frac{2}{t}\epsilon$. Hence, if $n$ is sufficiently large,  we get $N_{t+1,i+1}\geq n^{(t+1)\lambda_{t+1}+\epsilon}$, so $M_{i+1}$ truly satisfies the desired properties.
 \end{description}
 
 Therefore, we managed to construct the sequence of matrices $M_{0},M_{1},\dots, M_{z}$ satisfying properties 1. and 2. But then we arrive to a contradiction: the size of the matrix $M_{z}$ is $1\times  n$, but it contains more than $Cn^{t}$ copies of $K_{U,t}$ by Claim \ref{claim:analysis}, which is clearly impossible. This finishes the proof of Theorem \ref{thm:mainthm}.

\section{Ordered cycles}\label{sect:cycles}

If $A$ is a zero-one matrix whose corresponding ordered graph is a cycle, then call $A$ a \emph{cycle} as well. If $A$ is a zero-one matrix, connect two $1$-entries of $A$ by a segment if they are in the same row or column; call this the \emph{drawing of $A$}. More precisely, if $A(i,j)=1$, we imagine this $1$-entry as the point $(i,j)\in \mathbb{R}^{2}$, and so the drawing of $A$ is a subset of the plane composed of vertical and horizontal segments connecting the corresponding points of the $1$-entries. Note that if $A$ is a cycle, then the drawing of $A$ is a closed polygonal curve (possibly self-intersecting).

Now we  define two families of cycles. The matrix $A$ is an \emph{$x$-monotone cycle}, if $A$ is a cycle and every vertical line intersects at most two horizontal segments in the drawing of $A$. Also, $A$ is a \emph{positive cycle}, if the following holds. Direct the closed polygonal curve in the drawing of $A$, then $A$ is a positive cycle if every point on the plane is encircled a non-negative amount of times by the drawing of $A$ (we refer the reader to the paper of Pach and Tardos \cite{PT06} for a formal definition). We remark that every $2\times 2$-partite matrix that is a cycle is a positive cycle.

 Pach and Tardos \cite{PT06} noticed that zero-one matrices with no positive cycles correspond to incidence graphs of points and pseudo-lines, and they deduced the following theorem.

\begin{theorem}\label{thm:positivecycle}(Pach, Tardos \cite{PT06})
	If $\mathcal{C}$ is the family of positive cycles, then $\ex(n,\mathcal{C})=\Theta(n^{\frac{4}{3}})$.
\end{theorem} 

Hence, there are cycles of arbitrary length whose corresponding zero-one matrix has extremal number $\Omega(n^{\frac{4}{3}})$. However, Theorem \ref{thm:positivecycle} does not yield any upper bound on the extremal number of a single positive cycle $A$. In this section, we prove that for a large number of cycles their extremal number is $O(n^{\frac{3}{2}})$, which is sharp at least for the $2\times 2$ sized all-1 matrix. 

\begin{theorem}\label{thm:cycle}
	Let $A$ be an $x$-monotone cycle, then there exists a constant $c=c(A)$ such that  $$\ex(n,A)<cn^{\frac{3}{2}}.$$ 
\end{theorem}

The only $x$-monotone cycle that is also $2\times 2$-partite is the $2\times 2$ sized all-1 matrix, so in some sense Theorem \ref{thm:cycle} complements the $t=2$ case of Theorem \ref{thm:symmetric}. 

Let us briefly outline the idea of the proof. Let the size of $A$ be $r\times s$. Let $M$ be an $n\times n$ matrix of weight $\Omega(n^{\frac{3}{2}})$. Similarly as before, we divide $M$ into $k$ blocks along its rows (where $k$ is some large constant), and consider the distribution of $1$-entries in these blocks. We prove that either $M$ contains a matrix $N$ of large weight contained in the union of $r$ blocks, where the $1$-entries are equally distributed over the $r$ blocks in every column, or $M$ contains an $\frac{n}{k}\times\frac{n}{k}$ sized matrix $M'$ that is denser than $M$ in some sense. In the first case, we show that $A$ can be embedded into $N$ such that each row of $A$ is embedded into a different block. In the second case, we repeat the argument for $M'$. But then if $M$ does not contain $A$, we get a sequence of matrices of decreasing size but increasing density, and we arrive to a contradiction. Let us show how to execute this argument properly.

 As a reminder, the size of $A$ is $r\times s$. Note that we allow all-0 rows and columns in $A$, so $r$ might not be equal to $s$. Say that an $n\times m$ sized zero-one matrix $M$ is \emph{$r$-balanced}, if $r$ divides $n$ and for every column $c\in [m]$, the number of $1$-entries among $M((i-1)\frac{n}{r}+1,c),\dots,M(i\frac{n}{r},c)$ is the same for $i\in [r]$. In other words, if the column $c$ is divided into $r$ equal intervals, each interval contains the same number of $1$-entries. First, we show how to embed $A$ into an $r$-balanced matrix $M$ of sufficiently large weight.

\begin{lemma}\label{lemma:balanced}
	Let $M$ be an $n\times m$ sized $r$-balanced matrix. If 
	$$w(M)> rs\sqrt{m}n,$$
	then $M$ contains $A$.
\end{lemma}

\begin{proof}
	Say that a copy of $A$ in $M$ is \emph{proper}, if for each $j=1,\dots,r$, the $j$-th row of $A$ is embedded into one of the rows indexed by $(j-1)\frac{n}{r}+1,\dots,j\frac{n}{r}$ of $M$. We prove by induction on $s$ that $A$ has a proper copy in $M$. Without loss of generality, we can suppose that each column  of $A$ contains exactly two $1$-entries. 
	
	In the base case $s=2$, we prove that if $w(M)>\sqrt{m} n$, then $M$ contains a proper copy of $A$. Suppose that the $1$-entries of $A$ are contained in the rows indexed by $u$ and $v$. (That is, the only 1-entries of $A$ are $A(u,1)$, $A(v,1)$, $A(u,2)$ and $A(v,2)$.) Say that a triple $(a,b,c)\in [n]\times [n]\times [m]$ is \emph{good} if  $M(a,c)=M(b,c)=1$, $(u-1)\frac{n}{r}+1\leq a\leq u\frac{n}{r}$ and $(v-1)\frac{n}{r}+1\leq b\leq v\frac{n}{r}$. If there exists a pair $(a,b)$ and two different columns $c$ and $c'$ such that $(a,b,c)$ and $(a,b,c')$ are both good triples, then $M$ contains a proper copy of $A$. In total, there are $(\frac{n}{r})^{2}$ pairs $(a,b)$ such that $(u-1)\frac{n}{r}+1\leq a\leq u\frac{n}{r}$ and $(v-1)\frac{n}{r}+1\leq b\leq v\frac{n}{r}$, hence there are at most $(\frac{n}{r})^{2}$ good triples if $M$ does not contain a proper copy of $A$. If column $c$ contains $s_{c}$ $1$-entries, then there are $(\frac{s_{c}}{r})^{2}$ good triples containing $c$. Hence, the number of good triples is at least
	
	$$\sum_{c=1}^{m}\left(\frac{s_c}{r}\right)^{2}\geq \frac{1}{m}\left(\sum_{c=1}^{m}\frac{s_c}{r}\right)^{2}= \frac{1}{m}\left(\frac{w(M)}{r}\right)^{2}.$$
	Then $(\frac{n}{r})^{2} \ge \frac{1}{m}(\frac{w(M)}{r})^{2}$, a contradiction.

	Now let $s\geq 3$. Let $a,b\in [k]$ be the two indices such that $A(a,1)=A(b,1)=1$. As $A$ is $x$-monotone, either $A(a,2)=1$ or $A(b,2)=1$, otherwise the vertical line between the second and third columns intersects more than two horizontal lines in the drawing of $A$. Without loss of generality, let $A(a,2)=1$.
	
	Let $A'$ be the $r\times (s-1)$ sized matrix we get from $A$ by deleting the first column of $A$, replacing $A(a,2)$ with a $0$-entry and replacing $A(b,2)$ with a $1$-entry. Then $A'$ is also an $x$-monotone cycle, so by our induction hypothesis, every $r$-balanced $n\times m$ matrix with weight $r(s-1)\sqrt{m}n$ contains a copy of $A'$. Let $Q$ be the $n\times m$ matrix, whose $1$-entries correspond to those $1$-entries of $M$, that form $ A'(b, 1)$ in a proper copy of $A'$ in $M$. As $M$ has weight at least $rs\sqrt{m}n$, we have $w(Q)\geq \sqrt{m}n$. Indeed, suppose to the contrary that  there are less than $\sqrt{m}n$ such $1$-entries. Change these $1$-entries to $0$, and change at most $(r-1)\sqrt{m}n$ other $1$-entries to $0$ to make the matrix $r$-balanced. Then the resulting matrix still contains at least $r(s-1)\sqrt{m}n$ $1$-entries, so it contains a copy of $A'$, contradiction.
	
	Every $1$-entry of $Q$ is contained in the rows $(b-1)\frac{n}{r}+1,\dots,b\frac{n}{r}$, so by changing some of the $1$-entries of $M$ to $0$, we can find a $r$-balanced matrix $Q'$, which agrees on the rows $(b-1)\frac{n}{r}+1,\dots,b\frac{n}{r}$ with $Q$.
	
	Let $B$ be the $r\times 2$ sized matrix with $1$-entries $B(a,1),B(a,2),B(b,1),B(b,2)$. If $M$ contains a copy of $B$ and a copy of $A'$, where the $1$-entries corresponding to $B(b,2)$ and $A'(b,1)$ coincide, then $M$ contains $A$. But there are such copies because $Q'$ is an $r$-balanced matrix with weight at least $r\sqrt{m}n$, so it contains a proper copy of $B$. See Figure \ref{image:decomposing} for an illustration of $A'$ and $B$.
	\begin{figure}
	\begin{center}
		\begin{tikzpicture}
		\matrix[matrix of math nodes, left delimiter={(}, right delimiter={)},label=left:{$A=$~~~}] (A) at (-3,0) {
			1 & 1 &   &   \\
			1 &   &   & 1 \\
			&   & 1 & 1 \\
			& 1 & 1 &  \\
		};
		\draw[thick,opacity=.3] (A-1-2.center) -- (A-1-1.center) -- (A-2-1.center) -- (A-2-4.center) -- (A-3-4.center) -- (A-3-3.center) -- (A-4-3.center) -- (A-4-2.center) -- (A-1-2.center);
		
		\node at (-.5,0) {$\Rightarrow$};
		
		\matrix[matrix of math nodes, left delimiter={(}, right delimiter={)},label=left:{$A'=$~~~}] (B) at (2,0) {
			\textcolor{white}{1}  & \ &  \ \\
			1 & \  & 1 \\
			\  & 1 & 1 \\
			1 & 1 & \  \\
		};
		
		\draw[thick,opacity=.3] (B-2-1.center) -- (B-2-3.center) -- (B-3-3.center) -- (B-3-2.center) -- (B-4-2.center) -- (B-4-1.center) -- (B-2-1.center);  
		
		\node at (4.2,0) {and};
		
		\matrix[matrix of math nodes, left delimiter={(}, right delimiter={)},label=left:{$B=$~~~}] (C) at (7,0) {
			1 & 1 \\
			1 & 1 \\
			\textcolor{white}{1} & \ \\
			\textcolor{white}{1} & \ \\
		};
		\draw[thick,opacity=.3] (C-1-1.center) -- (C-1-2.center) -- (C-2-2.center) -- (C-2-1.center) -- (C-1-1.center);

		\end{tikzpicture}
	\end{center}
\caption{Decomposing $A$ into a copy of $K_{2,2}$ and an $x$-monotone cycle of shorter length.}
\label{image:decomposing}
\end{figure}
\end{proof}


\begin{lemma}\label{lemma:denseorbalanced}
	Let $k=2^{8}r^{2}$ and $c\geq 8rs\binom{k}{r}$. Let $M$ be an $n\times n$ sized matrix such that $k$ divides $n$ and $w(M)\geq cn^{\frac{3}{2}}$. Then either
	\begin{description}
		\item[(i)] $M$ contains an $\frac{n}{k}\times\frac{n}{k}$ sized matrix $N$ such that $w(N)\geq 2c(\frac{n}{k})^{\frac{3}{2}}$, or
		\item[(ii)] $M$ contains an $\frac{nr}{k}\times m$ sized $r$-balanced matrix $N$, where $w(N)\geq rs \sqrt{m} (\frac{nr}{k})$.	
	\end{description}
\end{lemma}

\begin{proof}[Proof of Lemma \ref{lemma:denseorbalanced}]
	Let $\alpha=\frac{1}{2k}$. Say that a column of $M$ is \emph{light} if it contains less than $\frac{cn^{\frac{1}{2}}}{2}$ $1$-entries. By deleting every light column, we get an $n\times m$ matrix $M'$ with $w(M')\geq \frac{cn^{\frac{3}{2}}}{2}$. For $i\in [k]$, let $M_{i}$ denote the submatrix of $M'$ formed by the rows indexed by $(i-1)\frac{n}{k}+1,\dots,i\frac{n}{k}$, and call $M_{i}$ a block. For $c\in [m]$ and $i\in [k]$, let $C_{c,i}$ denote the intersection of column $c$ with $M_{i}$, and let $s_{c}$ denote the number of $1$-entries in the column $c$. Call $C_{c,i}$ a column-block.
	
	Say that $(c,i)\in [m]\times [k]$ is \emph{sparse} if $C_{c,i}$ contains less than $\alpha s_{c}$ $1$-entries, otherwise it is \emph{dense}. Say that a column $c$ is \emph{imbalanced} if the number of indices $i\in [k]$ for which $C_{c,i}$ is dense is less than $r$, otherwise it is \emph{balanced}. If $c$ is imbalanced, then the number of $1$-entries in column $c$ contained in sparse column-blocks is at most $k\alpha s_{c}$, hence the number of $1$-entries contained in dense column-blocks is at least $(1-k\alpha)s_{c}$, which means that one of the  at most $r-1$ dense column-blocks contains at least $\frac{(1-k\alpha)s_{c}}{r}$ $1$-entries. Call such a column-block \emph{heavy}. Consider two cases.
	
	\textbf{Case 1.} The number of $1$-entries of $M'$  contained in imbalanced columns is at least $\frac{w(M')}{2}$. 
	
	In this case, we show that (i) holds. From each imbalanced column pick a heavy column-block. These heavy column blocks contain at least $\frac{w(M')(1-k\alpha)}{2r}$ $1$-entries. Therefore, there exists $i\in [r]$ such that the number of $1$-entries covered by heavy column-blocks in $M_{i}$ is at least $\frac{w(M')(1-k\alpha)}{2kr}$. Create the $\frac{n}{k}\times \frac{n}{k}$ sized matrix $N$ as follows: if there are less than $\frac{n}{k}$ heavy column-blocks in $M_{i}$, then let $N$ be the matrix formed by these column-blocks and some extra columns (so that the total number of columns is $\frac{n}{k}$); if there are at least $\frac{n}{k}$ heavy column-blocks in $M_{i}$, then pick any $\frac{n}{k}$ of them and let $N$ be their union. In the first case, we have $w(N)\geq \frac{w(M')(1-k\alpha)}{2kr}\geq \frac{cn^{\frac{3}{2}}(1-k\alpha)}{4kr}$. In the second case, we use that each column contains at least $\frac{cn^{\frac{1}{2}}}{2}$ $1$-entries, so each heavy column-block contains at least $\frac{cn^{\frac{1}{2}}(1-k\alpha)}{2r}$ $1$-entries. But then 
	$$w(N)\geq \frac{n}{k}\cdot\frac{cn^{\frac{1}{2}}(1-k\alpha)}{2r}=\frac{cn^{\frac{3}{2}}(1-k\alpha)}{2kr}.$$ Thus, in both cases, we have 
	$$w(N)\geq \frac{cn^{\frac{3}{2}}(1-k\alpha)}{4kr}=\frac{(1-k\alpha)ck^{\frac{1}{2}}}{4r}\left(\frac{n}{k}\right)^{\frac{3}{2}}.$$
	By the choice of $k$ and $\alpha$, this yields $w(N)\geq 2c(\frac{n}{k})^{\frac{3}{2}}$. 
	
	\textbf{Case 2.} The number of $1$-entries of $M'$ contained in balanced columns is at least $\frac{w(M')}{2}$. 
	
	In this case, we show that (ii) holds. For each balanced column $c\in [m]$,  let $I_{c}$ be a set of $r$ indices $i\in [k]$ such that $C_{c,i}$ is dense. In total, there are $\binom{k}{r}$ possible $r$ element sets in $[k]$, so there exists an $r$-element subset $I$ of $[k]$ and a set $C$ of columns such that for each $i\in I$ and $c\in C$, the column-block $C_{c,i}$ contains at least $\alpha s_{c}$ $1$-entries and $\sum_{c\in C}s_{c}\geq \frac{w(M')}{2\binom{k}{r}}$. Then, by changing some $1$-entries to $0$, if necessary, we find an $\frac{nr}{k}\times m$ sized $r$-balanced matrix $N$ in $M'$ such that
	$$w(N)\geq \frac{\alpha w(M')r}{2\binom{k}{r}}\geq \frac{\alpha cn^{\frac{3}{2}}r}{4\binom{k}{r}}\geq rs \sqrt{m} \left(\frac{nr}{k}\right).$$
\end{proof}
	
Combining the above two lemmas, we finish the proof of Theorem \ref{thm:cycle}.

\begin{proof}[Proof of Theorem \ref{thm:cycle}]
	Let $k=2^{8}r^{2}$ and $c\geq 8 rs\binom{k}{r}$. First, suppose that $n=k^{t}$, where $t$ is a positive integer. We prove that if $M$ is an $n\times n$ sized matrix such that $w(M)\geq cn^{\frac{3}{2}}$, then $M$ contains $A$. Suppose for a contradiction that $M$ does not contain $A$. Then we show that there exists a  sequence of zero-one matrices $M=M_{0},M_{1},\dots, M_{t}$ such that $M_{i-1}$ contains $M_{i}$, the size of $M_{i}$ is $\frac{n}{k^{i}}\times \frac{n}{k^{i}}$, and $w(M_{i})\geq 2^{i}c(\frac{n}{k^{i}})^{\frac{3}{2}}$ for $i=1,2,\dots$. However, this is clearly a contradiction as this yields that  $M_{t}$ is a $1\times 1$ sized matrix of weight at least $2^{t}c$, which is impossible.
	
	Indeed, if $M_{i-1}$ is already defined satisfying the above properties, define $M_{i}$ as follows.  We can apply Lemma \ref{lemma:denseorbalanced} to $M_{i-1}$ to get that either $M_{i-1}$ contains a matrix $M_{i}$ with the desired properties, or $M_{i-1}$ contains an $\frac{nr}{k^{i}}\times m$ sized $r$-balanced matrix $N$ with $w(N)\geq rs \sqrt{m} (\frac{nr}{k^{i}})$. But the latter case is impossible, because $N$ contains $A$ by Lemma \ref{lemma:balanced}.

	The only case remaining is when $n$ is not a power of $k$. Let $n_{0}$ be the smallest power of $k$ larger than $n$. Then clearly $n_{0}<kn$ and 
	$$\ex(n,A)\leq \ex(n_{0},A)\leq cn_{0}^{\frac{3}{2}}<ck^{\frac{3}{2}}n^{\frac{3}{2}},$$
	finishing the proof.
\end{proof}

\section{Concluding remarks}\label{sect:remarks}

As we mentioned in the introduction, the following strengthening of Theorem \ref{thm:mainthm} and Theorem \ref{thm:symmetric} remains open.

\begin{conjecture}\label{conj:maxdegree}
	Let $t\geq 2$ be a positive integer. Let $A$ be a zero-one matrix such that every row of $A$ contains at most $t$ $1$-entries. Then
	$$\ex(n,A)=n^{2-\frac{1}{t}+o(1)}.$$
\end{conjecture}

It would be already interesting to decide whether Conjecture \ref{conj:maxdegree} holds for every cycle $A$. An interesting open problem is to determine the order of magnitude of $\ex(n,A)$, where $A$ corresponds to a cycle of length $6$. There are six zero-one matrices corresponding to the $6$-cycle, see Figure \ref{image:6cycles}.

Each of these matrices is $x$-monotone, so we have $\ex(n,A)=O(n^{\frac{3}{2}})$ if $A$ corresponds to a cycle of length $6$. On the other hand, we have $\ex(n,A)=\Omega(n^{\frac{4}{3}})$ for all such matrices $A$ by a classical construction of Benson \cite{B66}. It would be interesting to close the gap between the lower and the upper bound. The following conjecture was also investigated in \cite{GyKMTTV18}.
 
\begin{conjecture}
	Let $A$ be a zero-one matrix that corresponds to a cycle of length $6$. Then $$\ex(n,A)=\Theta(n^{\frac{4}{3}}).$$
\end{conjecture}  

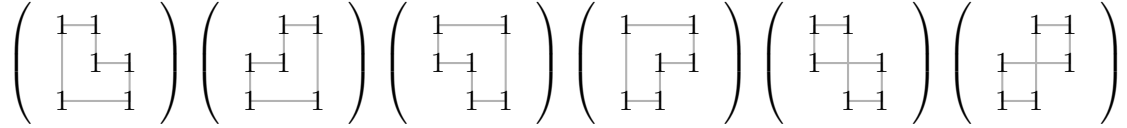
\begin{figure}
	\begin{center}
		\begin{tikzpicture}
		\matrix[matrix of math nodes, left delimiter={(}, right delimiter={)}] (A) at (-2.5,0) {
			1 & 1 &   \\
			& 1 & 1 \\
			1 &   & 1 \\
		};
		\draw[thick,opacity=.3] (A-1-1.center) -- (A-1-2.center) -- (A-2-2.center) -- (A-2-3.center) -- (A-3-3.center) -- (A-3-1.center) -- (A-1-1.center);
		
		\matrix[matrix of math nodes, left delimiter={(}, right delimiter={)}] (B) at (0,0) {
			& 1 & 1 \\
			1	& 1 &   \\
			1 &   & 1 \\
		};
		\draw[thick,opacity=.3] (B-2-1.center) -- (B-2-2.center) -- (B-1-2.center) -- (B-1-3.center) -- (B-3-3.center) -- (B-3-1.center) -- (B-2-1.center);
		
		\matrix[matrix of math nodes, left delimiter={(}, right delimiter={)}] (C) at (2.5,0) {
			1 &   & 1 \\
			1 & 1 &   \\
			& 1 & 1 \\
		};
		\draw[thick,opacity=.3] (C-1-1.center) -- (C-1-3.center) -- (C-3-3.center) -- (C-3-2.center) -- (C-2-2.center) -- (C-2-1.center) -- (C-1-1.center);
		
		\matrix[matrix of math nodes, left delimiter={(}, right delimiter={)}] (D) at (5,0) {
			1 &   & 1 \\
			& 1 & 1 \\
			1 & 1 &   \\
		};
		\draw[thick,opacity=.3] (D-1-1.center) -- (D-1-3.center) -- (D-2-3.center) -- (D-2-2.center) -- (D-3-2.center) -- (D-3-1.center) -- (D-1-1.center);
		
		\matrix[matrix of math nodes, left delimiter={(}, right delimiter={)}] (E) at (7.5,0) {
			1 & 1 &   \\
			1 &   & 1 \\
			& 1 & 1 \\
		};
		\draw[thick,opacity=.3] (E-1-1.center) -- (E-1-2.center) -- (E-3-2.center) -- (E-3-3.center) -- (E-2-3.center) -- (E-2-1.center) -- (E-1-1.center);
		
		\matrix[matrix of math nodes, left delimiter={(}, right delimiter={)}] (F) at (10,0) {
			& 1 & 1 \\
			1 &   & 1 \\
			1 & 1 &   \\
		};
		\draw[thick,opacity=.3] (F-1-2.center) -- (F-1-3.center) -- (F-2-3.center) -- (F-2-1.center) -- (F-3-1.center) -- (F-3-2.center) -- (F-1-2.center);
		
		\end{tikzpicture}
	\end{center}
	\caption{The six matrices corresponding to cycles of length 6.}
	\label{image:6cycles}
\end{figure}

	
\section*{Acknowledgment}
We would like to thank J\'anos Pach and G\'abor Tardos for valuable discussions. The contents of Section \ref{sect:cycles} also appear in the master thesis of Francesco Calignano from EPFL. These results originate from the second author of this paper, who was also the supervisor of this master thesis.

%

%
	
\end{document}